\documentclass[10pt]{amsart}
\usepackage{amssymb,mathrsfs}
\usepackage{mathabx}
\usepackage{color}
\usepackage{graphicx}
\usepackage{float}
\usepackage[bottom]{footmisc}
\usepackage[colorlinks,hyperindex,linkcolor=blue]{hyperref}
\hypersetup{
 pdfauthor = {Roger Casals, Alvaro del Pino Gomez, Francisco Presas},
 pdftitle= {Loose Engel structures},
 pdfsubject = {Differential geometry, h--principle, distributions, Engel structures}}

\newcommand{\SE}{{\mathcal{E}}}
\newcommand{\SD}{{\mathcal{D}}}
\newcommand{\SW}{{\mathcal{W}}}
\newcommand{\SY}{{\mathscr{Y}}}

\newcommand{\SX}{{\mathfrak{X}}}
\newcommand{\ST}{{\mathcal{T}}}
\newcommand{\SU}{{\mathcal{U}}}
\newcommand{\SV}{{\mathcal{V}}}

\newcommand{\Engel}{\operatorname{\mathscr{E}}}
\newcommand{\FEngel}{{\operatorname{\mathscr{E}^f}}}
\newcommand{\OT}{{\operatorname{OT}}}
\newcommand{\Loose}{{\operatorname{\mathscr{L}}}}
\newcommand{\Tight}{{\operatorname{Tight}}}

\newcommand{\SL}{{\mathcal{L}}}
\newcommand{\Imm}{{\operatorname{\mathcal{I}}}}
\newcommand{\FImm}{{\operatorname{\mathcal{I}^f}}}

\newcommand{\R}{{\mathbb{R}}}
\newcommand{\Z}{{\mathbb{Z}}}
\newcommand{\NS}{{\mathbb{S}}}
\newcommand{\D}{{\mathbb{D}}}
\newcommand{\N}{\mathbb{N}}
\newcommand{\PP}{\mathbb{P}}

\newcommand{\Cont}{\operatorname{\mathscr{C}}}
\newcommand{\FCont}{\operatorname{\mathscr{C}^f}}

\newcommand{\Op}{{\mathcal{O}p}}

\newcommand{\no}{{\operatorname{n.o.}}}
\newcommand{\Mon}{{\operatorname{Mon}}}

\newcommand{\acts}{{\righttoleftarrow}}

\newtheorem{lemma}{Lemma}
\newtheorem{proposition}[lemma]{Proposition}
\newtheorem{theorem}[lemma]{Theorem}
\newtheorem{corollary}[lemma]{Corollary}
\newtheorem{definition}[lemma]{Definition}

\theoremstyle{remark}
\newtheorem{remark}[lemma]{Remark}

\begin{document} 

\title{Loose Engel structures}

\subjclass[2010]{Primary: 58A17, 58A30.}
\date{\today}

\keywords{Engel structures, $h$--principle, flexibility}

\author{Roger Casals}
\address{University of California Davis, Department of Mathematics, Shields Avenue, Davis, CA 95616, USA}
\email{casals@math.ucdavis.edu}

\author{\'Alvaro del Pino}
\address{Utrecht University, Department of Mathematics, Budapestlaan 6, 3584 Utrecht, The Netherlands}
\email{a.delpinogomez@uu.nl}

\author{Francisco Presas}
\address{Instituto de Ciencias Matem\'aticas -- CSIC. C. Nicol\'as Cabrera, 13--15, 28049 Madrid, Spain.}
\email{fpresas@icmat.es}

\begin{abstract}
This article contributes to the study of Engel structures and their classification. The main result introduces the notion of a loose family of Engel structures and shows that two such families are Engel homotopic if and only if they are formally homotopic. This implies a complete $h$--principle when auxiliary data is fixed. As a corollary, we show that Lorentz and orientable Cartan prolongations are classified up to homotopy by their formal data. 
\end{abstract}

\maketitle

\section{Introduction}

Let $M$ be a smooth $n$-dimensional manifold. By definition, an $m$-dimensional smooth distribution $E\subset TM$ is a smooth section of the Grassmannian bundle $\mbox{Gr}_m(TM)\longrightarrow M$. Distributions are a core geometric structure in the modern perspective of differential geometry and control theory \cite{BH93,GV,Mon2,MZ09}, which in particular subsumes the smooth dynamics of non-vanishing vector fields, the theory of foliations, and contact geometry. The integrability of geometric structures, including the existence of complex structures, and the symmetries (and solvability) of differential equations are part of the theory of distributions \cite{Cartan1901,Car,GV}.

\'E.~Cartan addressed in \cite{Cartan1901} the existence of a local normal form for a generic distribution, i.e.~the lack of local invariants when the distribution is given by an open condition. The main result \cite{Cartan1901,Mon,Mon3} is that a generic distribution $E\subset TM$ has a unique local normal form if and only if it belongs to one of the following families: smooth line fields, contact structures (even or odd), or Engel structures. The study of the first two geometries, smooth dynamics and contact topology, have been topics of major interest and activity in the last four decades. 

Engel structures, maximally non-integrable 2-distributions in 4-manifolds, have proven themselves more elusive: in \cite[Intrigue F]{EM}, Y. Eliashberg and E. Mishachev identified the classification of Engel structures as an outstanding problem in the theory of $h$-principles. The first modern breakthrough in Engel geometry was the existence theorem proven by T.~Vogel \cite[Theorem 6.1]{Vo0}. In the last two years, the study of Engel structures has further seen significant developments \cite{CP,CPPP,CPV,KV18,Mitsumatsu18,Pia,dP17,PV,Yamazaki,Zhao18,ZhaoThesis18}, exhibiting unique properties of Engel structures and connections with smooth dynamics and contact and symplectic geometry. In particular, the authors proved the parametric existence $h$-principle for Engel structures in \cite{CPPP}, and the first two authors developed the complete $h$-principle for non-singular Engel knots in \cite{CP}.

The present article continues this work by providing a classification $h$-principle for a class of Engel structures called {\it loose} (see Section \ref{ssec:mainDef} for a definition). Our result should be compared (see the Appendix) to the work in \cite{PV}, which defines and classifies another class of Engel structures, called \emph{overtwisted}. The later is closer in behaviour to the class of {\it overtwisted} contact structures, which were shown to satisfy a complete $h$-principle in \cite{BEM}.

Let us now state the main theorems of this article in precise terms.

\subsection{Main Results}

Let $M$ be a closed smooth $4$-manifold, $\Engel(M)$ the space of Engel structures on $M$, and $\FEngel(M)$ its formal counterpart \cite{CPPP,Gr86,PV} (see Subsection \ref{ssec:EngelStructures} for a definition). It was proven in \cite{CPPP} that the scanning map given by the inclusion
$$ \Engel(M) \longrightarrow \FEngel(M) $$
induces a surjection in homotopy groups. 

The aim of the present work is to show that every homotopy class in $\pi_k(\FEngel(M))$ can be represented by a $k$--dimensional sphere in $\Engel(M)$ which is unique up to Engel homotopy, i.e.~there is a subgroup $\Loose_k(M) \subset \pi_k(\Engel(M))$ isomorphic to $\pi_k(\FEngel(M))$ that can be characterised in a geometric fashion. This is the content of our two main results:

\begin{theorem} \label{thm:main1}
Let $M$ be a smooth $4$-manifold, $K$ a compact CW--complex, and $N$ a positive integer. Then, any family $\SD: K \longrightarrow \FEngel(M)$ is formally homotopic to an $N$--loose family.
\end{theorem}

We strengthen the existence $h$-principle in Theorem \ref{thm:main1} to the following uniqueness $h$-principle: 

\begin{theorem} \label{thm:main2}
Let $M$ be a smooth $4$-manifold and $K$ a compact CW--complex. There exists a positive integer $N_0$, depending only on $\dim(K)$, such that: Any two $N$--loose families $\SD_0, \SD_1: K \longrightarrow \Engel(M)$, $N \geq N_0$, are Engel homotopic if they are formally homotopic.

In addition, the resulting Engel homotopy $(\SD_t)_{t \in [0,1]}$ can be realised as a $(N-N_0)$--loose $K \times [0,1]$--family of Engel structures.
\end{theorem}

Theorem \ref{thm:main1} provides existence and Theorem \ref{thm:main2} shows uniqueness. The notion of looseness for a family of Engel structures will be introduced in Definition \ref{def:loose}, Section \ref{sec:main}. Roughly, it is a quantitative property which measures the rotation of the Engel plane field $\SD$ with respect to a line field $\SY \subset \SD$, captured by a positive integer $N$. In particular, if a family of Engel structures $\SD: K \longrightarrow \FEngel(M)$ is $N_2$--loose, then it is $N_1$--loose for any $N_1 \leq N_2$. By definition, the line field $\SY$ is called the \emph{certificate} and a family that is $N$--loose with $N \geq N_0$, with $N_0$ as in the statement of Theorem \ref{thm:main2}, is said to be simply \emph{loose}.

Let $\FEngel(M,\SY)$ be the space of formal Engel structures containing some fixed $\SY\subset \SD$ transverse to the formal kernel $\SW$. If $\SY$ has no periodic orbits, Theorem \ref{thm:main1} can be strengthened to yield families that are $N$-loose for all $N$. Such a family is said to be $\infty$--loose. We denote by $\Loose(M,\SY)$ the subspace of Engel structures having $\SY$ as their certificate of $\infty$--looseness. Using Theorems \ref{thm:main1} and \ref{thm:main2} we can deduce the following complete $h$--principle:
\begin{theorem}\label{thm:main}
Let $M$ be a closed smooth $4$-manifold and $\SY$ a line field without periodic orbits. Then, the forgetful inclusion $\Loose(M,\SY) \longrightarrow \FEngel(M,\SY)$ is a weak homotopy equivalence.
\end{theorem}

In Section \ref{sec:app} we compare these statements with other recent developments regarding flexibility in Engel topology. In Subsection \ref{ssec:EngelProlong} we define the notion of Cartan/Lorentz prolongation and we prove:
\begin{corollary} \label{cor:prolong}
Any family of Lorentz or orientable Cartan prolongations is loose, up to Engel homotopy. In particular, such families are classified, up to Engel homotopy, by their formal data.
\end{corollary}
Which subsumes one of the main results in \cite{dP17}. In Subsection \ref{ssec:cppp}, we prove that the Engel structures produced in \cite{CPPP} are homotopic to loose ones, and that that those constructed using Engel open books in \cite{CPV} are loose.

The article is organised as follows: Section \ref{sec:preliminaries} is dedicated to convexity in Engel topology, including all the basic theory needed for our results. Section \ref{sec:main} defines and explores Engel looseness. The proof of Theorems \ref{thm:main1}, \ref{thm:main2}, and \ref{thm:main} is structured in two parts: existence of loose families (Subsection \ref{ssec:existence}) and uniqueness (Subsection \ref{ssec:uniqueness}). Section \ref{sec:app} contains applications, including the proof of Corollary \ref{cor:prolong}. Section \ref{sec:appendix} provides a detailed discussion comparing flexibility phenomena for Engel structures and contact structures. Particularly, we discuss the relation between the present article and the work in \cite{PV}.

\textbf{Acknowledgements.} The authors are grateful to the American Institute of Mathematics and to the organisers and participants of the workshop ``Engel structures'' for their interest in this work. They are particularly thankful to T. Vogel for reading a preliminary version of these results. The authors are also grateful to the referees of the article, whose suggestions and comments have greatly improved the present work. The authors are supported by Spanish National Research Projects MTM2016--79400--P, MTM2015-72876-EXP and SEV2015-0554. R.~Casals is supported by the NSF grant DMS-1841913 and a BBVA Research Fellowship. \'A.~del Pino is supported by the NWO Vici Grant no. 639.033.312 of Marius Crainic.

\section{Engel structures and convexity} \label{sec:preliminaries}

In this section we state the basic facts and techniques used in the study of Engel structures. We focus on the interaction between Engel structures and families of convex curves in the $2$--sphere. This relationship will allow us to prove Theorems \ref{thm:main1}, \ref{thm:main2}, and \ref{thm:main}.

We use the notation $\Op(A)$ to denote an arbitrarily small neighbourhood of the subset $A$.

\subsection{Engel structures} \label{ssec:EngelStructures}

The central objects of study are the following geometric structures:

\begin{definition}
An \textbf{Engel structure} is a maximally non--integrable $2$--plane field $\SD\subset TM$. That is, $\SE = [\SD,\SD]$ is an everywhere non--integrable $3$--distribution, i.e. $TM = [\SE,\SE]$.

The distribution $\SE$ is said to be an \textbf{even-contact structure}. It contains a line field $\SW$ uniquely defined by the equation $[\SW,\SE] \subset \SE$. The line field $\SW$ is said to be the \textbf{kernel} of $\SE$.
\end{definition}

It follows from its definition that the line field $\SW$ is contained in the Engel structure $\SD$. In consequence, an Engel structure $\SD$ induces a complete flag $\SW \subset \SD \subset \SE$ on the 4-manifold $M$ \cite{CPPP}. In addition, the Lie bracket induces two canonical bundle isomorphisms:
\begin{align} \label{eq:iso1}
\det(\SD) \cong \SE/\SD,
\end{align}
\begin{align} \label{eq:iso2}
\det(\SE/\SW) \cong TM/\SE.
\end{align}

Decoupling the differential relation that determines Engel structures leads us to define their formal counterpart as follows: a \textbf{formal Engel structure} is a complete flag $\SW \subset \SD \subset \SE\subset TM$ endowed with bundle isomorphisms as in Equations (\ref{eq:iso1}) and (\ref{eq:iso2}). In this case, there is no differential relationship between the different distributions that constitute the flag. We will often refer to $\SW$ as the \emph{formal kernel} of the \emph{formal even-contact structure} $\SE$.

Let $\Engel(M)$ be the space of Engel structures endowed with the $C^0$-topology, and $\FEngel(M)$ the space of formal Engel structure endowed with the $C^2$-topology. The present work focuses on the homotopy theoretic nature of the inclusion
$$s:\Engel(M) \longrightarrow \FEngel(M).$$
This forgetful inclusion is continuous with the chosen topologies. This map is classically called the {\it scanning map} \cite{CPPP,Gr86} and it is the main focus in the study of $h$-principles \cite{EM}.

\subsection{Engel flowboxes and convexity} \label{ssec:EngelCharacterisation}

Let us explain a useful method to construct Engel structures locally. For reference, a $2$-plane in a smooth $3$-manifold is maximally non-integrable, i.e.~a contact structure, if and only if the contact planes strictly rotate with respect to a foliation by Legendrian lines \cite{Ge}. In the same vein, the Engel condition can be geometrically described in terms of a flowbox for a line field contained in the Engel structure, as follows.

Fix coordinates $(p,t)$ in the product $\D^3 \times [0,1]$ and consider the bundle isomorphism
\[ d_{(p,t)}\pi: T_{(p,t)}(\D^3 \times \{t\}) \longrightarrow T_p\D^3, \]
where $\pi: \D^3 \times [0,1] \longrightarrow \D^3$ is the projection onto the first factor. Any given fibrewise identification of the projectivized bundle $\PP(T\D^3)$ with $\R\PP^2$ obtained by fixing a framing of $T\D^3$ can be lifted to an identification
\[ d_{(p,t)}\pi: \PP T_{(p,t)}(\D^3 \times \{t\}) \longrightarrow \R\PP^2. \]

We focus on the $2$--distributions $\SD$ of the form $\langle \partial_t, X(t) \rangle$, with $X(t)$ a vector field tangent to the slice $\D^3 \times \{t\}$. The vector field $X$ can be regarded as a $\D^3$--family of curves 
\[ X_p: [0,1] \longrightarrow \R\PP^2, \]
\[ X_p(t) = d_{(p,t)}\pi([X(p,t)]), \]
where $[-]$ denotes the associated line. The characterization of the Engel condition for $\SD$ then reads:

\begin{proposition}[\cite{CPPP}] \label{prop:EngelCharacterisation}
The module $\SE = [\SD,\SD]$ is a $3$-distribution on $\Op(p,t)$ if and only if the curve $X_p$ is immersed at time $t$.

The $2$-plane field $\SD$ is an Engel structure on a neighbourhood of the point $(p,t)$ if and only if, additionally, at least one of the following two conditions holds:
\begin{enumerate}
  \item[A.] The curve $X_p(t)$ has no inflection point at time $t$,
  \item[B.] The $2$--distribution $\langle X(q,t),X'(q,t)\rangle$ is a contact structure in $\Op(p)\times\{t\}$.
\end{enumerate}

If $\SD$ is Engel, its kernel $\SW$ is spanned by $\partial_t$ at the point $(p,t)$ if and only if the curve $X_p$ has an inflection point at time $t$.\hfill$\Box$
\end{proposition}
By definition, $t$ is an \emph{inflection point} for the curve $X_p(t)$ if $X_p$ has a tangency at $t$ of order at least $2$ with the great circle $\langle X_p(t), X_p'(t) \rangle$. We will focus on assumption (A), i.e. the curves $X_p$ will be everywhere convex (or concave); in particular, $\partial_t$ will be transverse to the kernel.

\begin{remark} \label{rem:CPPP}
The techniques developed in \cite{CPPP} are based on the interaction between conditions (A) and (B). For completeness, we prove in Subsection \ref{ssec:cppp} that the families constructed in \cite{CPPP} are loose. \hfill$\blacksquare$
\end{remark}

\subsection{Convex curves and little wiggles} \label{ssec:convex}

Proposition \ref{prop:EngelCharacterisation} connects the study of Engel structures with the theory of convex curves in $\R\PP^2$. The classical results in this direction \cite{Lit,Sal} are stated for convex curves into the 2-sphere $\NS^2$, but they easily translate to the $\R\PP^2$ setting. Let us explain this in detail.

Fix a $1$-manifold $I$. Let $\Imm(I)$ be the space of immersions of $I$ into $\R\PP^2$, endowed with the $C^1$--topology. Consider the space $\FImm(I)$ of formal immersions of $I$ into $\R\PP^2$, endowed with the $C^0$--topology, and the subspace $\SL(I)\subset\Imm(I)$ of locally convex curves, endowed with the $C^2$--topology. The inclusion of $\SL(I)$ into $\Imm(I)$ is continuous and the formal counterpart of $\SL(I)$ is homotopy equivalent to $\FImm(I)$ \cite{EM,Gr86}.

The following notion will be important to us:
\begin{definition} \label{def:wiggle}
A curve $g \in \Imm(I)$ has a \textbf{wiggle} in the interval $[a,b] \subsetneq I$ if $g(a) = g(b)$, and, after identifying endpoints, $g|_{[a,b]}$ is a smooth closed convex embedded curve.

The curve $g$ has $n$ wiggles if there are $n$ intervals $\{I_i \subsetneq I\}_{i=1}^n$, each of them a wiggle of $g$. We require the interiors of these intervals to be pairwise disjoint, but we allow their endpoints to agree. When this happens, we say that the wiggles are concatenated.
\end{definition}

\begin{remark} \label{rem:wiggle}
Suppose $n$ wiggles $\{I_i\}_{i=1}^n$ are concatenated. Then, one may be able to choose some other interval $I' \subset I$ different from them, contained in their union, which is also a wiggle. Due to the embeddedness of a wiggle, we deduce that $I'$ is either one of the $I_i$ or it intersects exactly two of the original intervals. \hfill$\blacksquare$
\end{remark}

We depict wiggles in several figures. For clarity, we often do so up to a small homotopy through convex curves.

\subsubsection{Adding wiggles to curves} \label{sssec:wiggle}

Let $K$ be a compact manifold, $n\in\N$ a positive integer, and fix maps $f: K \longrightarrow \Imm(I)$ and $t: K \longrightarrow I$. From this data, we construct a new map
$$ f^{[n\# t]}: K \longrightarrow \Imm(I) $$
as follows: For each $k \in K$, we cut the curve $f(k)$ at the point $f(k)(t(k))$ and we add $n$ small convex loops, smoothing the result. This defines the map $f^{[n\# t]}$, which has (as least) $n$ concatenated wiggles. We can and do assume that the two maps $f$ and $f^{[n\# t]}$ agree as parametrised curves outside of an arbitrarily small neighbourhood of the inserted loops. The insertion of wiggles can be done over different points as long as we have functions $t_0,\dots,t_m: K \longrightarrow I$ with disjoint images; we then write $f^{[n_0\#t_0,\dots,n_m\#t_m]}$ for the resulting family.

Work of J.A. Little \cite{Lit} implies that the space of wiggles passing through a point with a given direction is contractible. Therefore, our cutting process is unique up to a convex homotopy of the added wiggle.

\begin{remark}
Wiggles which are concatenated may have different images (unlike in the construction just provided). However, invoking again the contractibility result due to Little, we deduce that we can make the images be the same by a homotopy through convex curves. \hfill$\blacksquare$
\end{remark}

\subsubsection{Achieving convexity}

The purpose of adding wiggles is that they provide convexity. The first ingredient we need is that any immersed curve $f \in \Imm(I)$ is homotopic to a curve with sufficiently many wiggles:
\begin{lemma}[\cite{Lit,Sal}]\label{lem:addingWiggles}
Let $f: K \longrightarrow \Imm(I)$ be a $K$--family of immersed curves and $t_0 \in I$. Then, the families $f^{[n_0\#t_0]}$ and $f^{[n_0+2\#t_0]}$ are homotopic through immersions. The homotopy can be assumed to have support in a small neighborhood $(t_0-\varepsilon , t_0+\varepsilon)$ of the cutting point $t_0$. \hfill$\Box$
\end{lemma}
The homotopy of immersed curves described in Lemma \ref{lem:addingWiggles} is shown in Figure \ref{fig:immersions}.

\begin{figure}[ht] 
\centering
\includegraphics[scale=0.27]{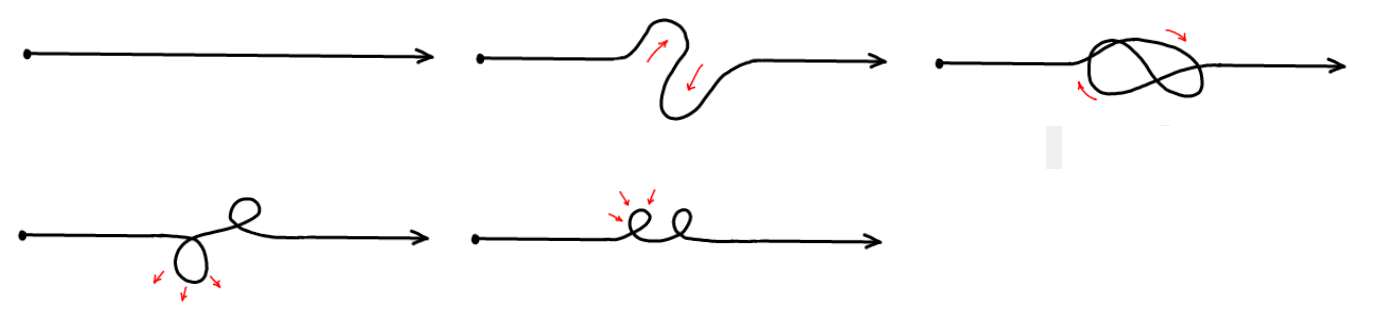}
\caption{Homotopy of immersed curves where two wiggles are introduced. In the last step, we take the concave wiggle and we push it around $\R\PP^2$ so that it appears as a convex wiggle.}
\label{fig:immersions}
\end{figure}

The key result in \cite{Lit}, explained in detail in \cite[Section 6]{Sal}, states that if $f$ is already convex and one extra loop is introduced, additional wiggles may be added using a homotopy through convex curves:
\begin{lemma}[\cite{Lit,Sal}]\label{lem:convexity}
Let $f: K \longrightarrow \SL(I)$ be a $K$--family of convex curves and $t_0 \in I$. Then, the families $f^{[n_0\#t_0]}$ and $f^{[n_0+2\#t_0]}$ are homotopic through convex curves as soon as $n_0>0$. The homotopy can be assumed to have support in a small neighborhood $(t_0-\varepsilon , t_0+\varepsilon)$ of $t_0$ containing the existing wiggle. \hfill$\Box$
\end{lemma}
This homotopy of convex curves from Lemma \ref{lem:convexity} is shown in Figure \ref{fig:Little}; we refer to it as \emph{Little's homotopy}.

\begin{figure}[ht] 
\centering
\includegraphics[scale=0.2]{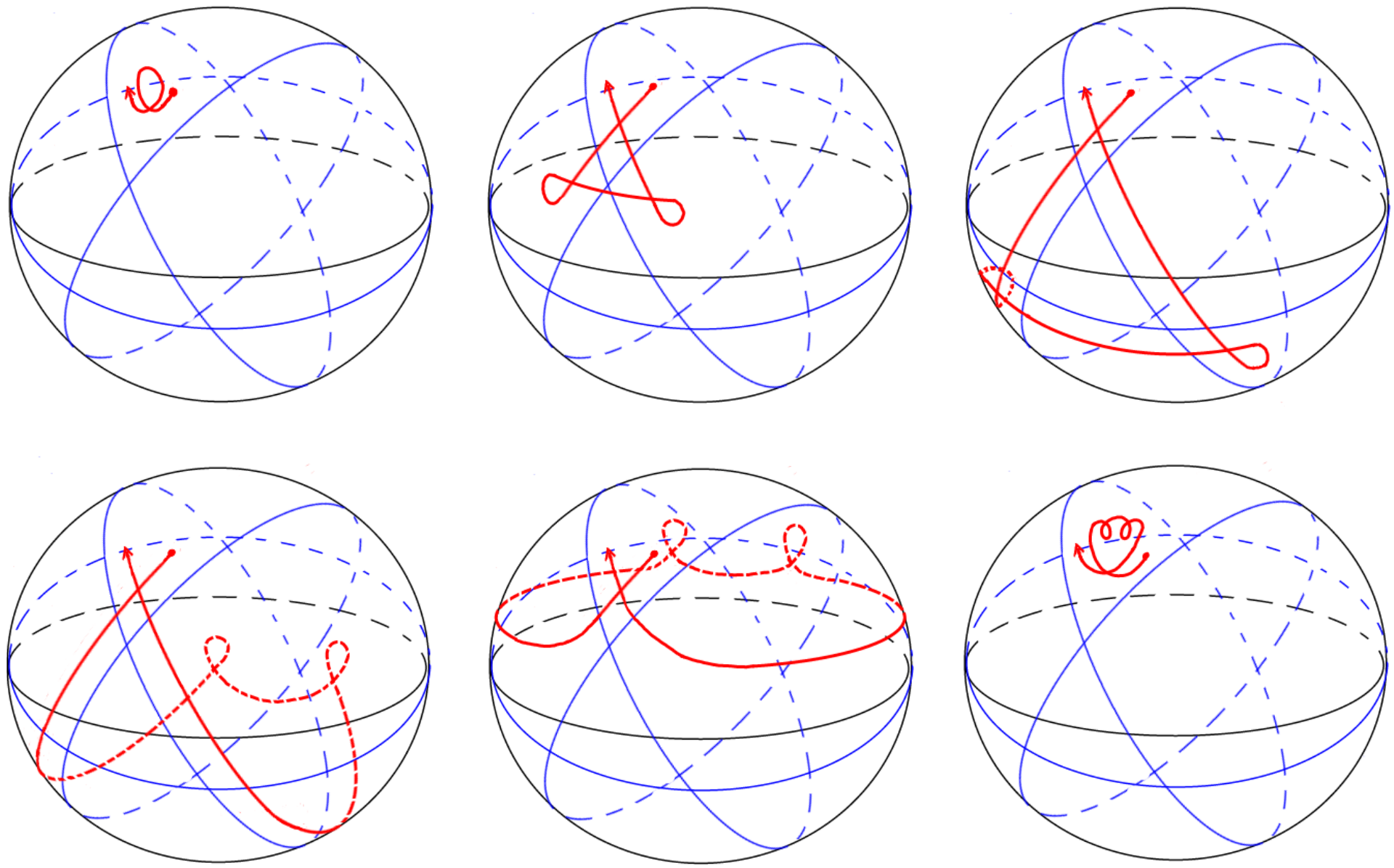}
\caption{Little's homotopy for convex curves in $\NS^2$. The closed curves correspond to great circles. The first figure shows a convex curve with a little wiggle. By pushing the wiggle down, it can be taken to the second figure. It is convex because it is comprised of three segments that are slight push-offs of equators whose corners have been rounded to preserve convexity. The same is true for the third and fourth figures. In the last two images we push towards the opposite hemisphere, yielding a curve with three wiggles. As shown, this process is relative, in the domain, to the complement of a small neighbourhood of the wiggle.}
\label{fig:Little}
\end{figure}

The last remark we need is: Given any curve $f \in \Imm(I)$ and any sufficiently dense collection of points $t_0,\dots,t_m: K \longrightarrow I$, the curve $f^{[1\#t_0,\dots,1\#t_m]}$ will be homotopic to a convex curve. Furthermore, the $C^1$-size of the homotopy needed to achieve convexity will be inversely proportional to the number of wiggles introduced. This is the content of the following Proposition, which is the crucial geometric ingredient behind Theorem \ref{thm:main}.
\begin{proposition} \label{prop:slidingWiggles}
Let $K$ be a compact manifold, $A \subset K$ a closed submanifold, and $0<a<1/2$ a positive real constant. Suppose that $f: K \longrightarrow \Imm([0,1])$ is a family of immersions such that $f(A) \subset \SL([0,1])$ and there exists $F: K \longrightarrow \SL([0,a]\cup[1-a,1])$ such that $f(k)(t) = F^{[1\#a]}(k)(t)$.

Then, there exists a family $f: K \times [0, \infty) \longrightarrow \Imm([0,1])$ such that:
\begin{itemize}
\item[-] For $s$ large enough $f(k,s)$ is everywhere convex,
\item[-] $f(k,s)(t) = f(k)(t)$ if $s=0$, $k \in A$, or $t \in [0,a/2] \cup [1-a,1]$,
\item[-] The number of wiggles of $f(k,s)$ in $[a,1-a]$ goes to infinity as $s \longrightarrow \infty$ if $k \notin \Op(A)$. The maximum distance between two consecutive wiggles in this segment is $O(1/s)$ and the radius of each wiggle is $O(1/s)$.
\end{itemize}
\end{proposition}

\begin{figure}[ht]
\centering
\includegraphics[scale=0.45]{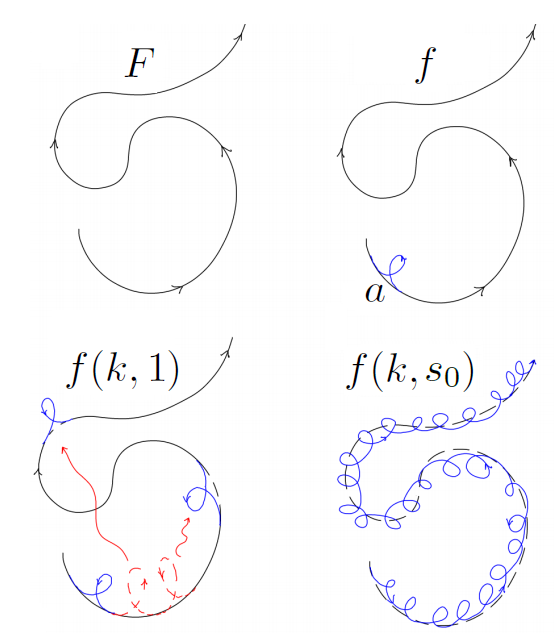}
\caption{We use Little's homotopy to create several wiggles from a given one. These wiggles are then distributed along the curve to achieve convexity everywhere.}
\label{fig:slidingWiggles}
\end{figure}

Figure \ref{fig:slidingWiggles} depicts the content of Proposition \ref{prop:slidingWiggles}. The density of wiggles goes to infinity as $s$ does, as we will see in the proof, while their size has order $O(1/s)$.

\begin{proof}[Proof of Proposition \ref{prop:slidingWiggles}]
We construct $f(k,s)$ in each interval $s=(n,n+1)$ by induction on $n\in\N$. For $s \in [n, n+1]$ and $k$ in the complement of an arbitrarily small neighborhood $\Op(A)$ of $A$, we obtain the following properties:
\begin{itemize}
\item[-] $f(k,n)= f(k)^{[1\#a, 1\#(a+\frac{1}{2n}(1-2a)), \ldots, 1\#(a+\frac{2n-1}{2n}(1-2a)), 1\#(1-a)]}$.
\item[-] $f(k,n+1)= f(k)^{[1\#a, 1\#(a+\frac{1}{2n+2}(1-2a)), \ldots, 1\#(a+\frac{2n+1}{2n+2}(1-2a)), 1\#(1-a)]}$.
\item[-] There are $2n$ paths sliding the last $2n$ wiggles of $f(k,n)$ to the last $2n$ wiggles of $f(k, n+1)$. We change the insertion points from $t=a+\frac{j}{2n}(1-2a)$ to $t=a+\frac{j+2}{2n+2}(1-2a)$, $j=1,\ldots, 2n$, by linear interpolation.
\item[-] The radius of those $2n$ wiggles is exactly $1/s$.
\item[-] In $Op(A)$, $f(k,s)$ remains convex. Moreover, $f(k,s)=f(k,0)$ for $k\in A$.
\end{itemize}

Indeed, this is simple to build: the only geometrically non--trivial part corresponds to the family in the interval $t\in (a, a + \frac{2}{2n}(1-2a))$. In the parameter interval $s\in (n, n+1/2)$ we use Little's homotopy to produce three wiggles out of the existing wiggle at time $t=a$, always ensuring that they have radius $1/n$. For $s\in(n+1/2, n+1)$, we linearly move the insertion points to place them at times $a$, $a+ \frac{1}{2n+2}(1-2a)$ and $a+ \frac{2}{2n+2}(1-2a)$ for $s=n+1$. Since convexity is preserved during Little's homotopy, we can assume that the insertion of the additional wiggles is done relative to $A$ by cutting--off Little's homotopy for $k \in \Op(A) \setminus A$. 

To conclude the argument we need to distribute the convexity of the wiggles to make $f(k,s)$ everywhere convex for $s$ large enough. For that, use only half of the wiggles to create convexity, i.e. the ones placed in even positions with respect to the order provided by the insertion time. As explained before, this makes the new family convex for $s$ large. Moreover, the odd wiggles are unaffected by this process. Therefore, the number of wiggles is $O(s)$ and they are uniformly distributed.
\end{proof}

\subsection{The development map} \label{ssec:developmentMap}

Let us now introduce the notion of \emph{development map}, which is used in order to define loose Engel structures. Geometrically, the development map allows us to intrinsically describe the turning of an Engel structure $\SD$ with respect to a line field $\SY$ contained inside it. Note that this time we are not resorting to the use of charts/flowboxes, as in Subsection \ref{ssec:EngelCharacterisation}. The development map is well-known in the particular case where the line field is the kernel, and under this assumption it was first studied by R. Montgomery \cite{Mon}.

Since the development map encodes how the 2-plane $\SD$ moves along $\SY$ in terms of the linear holonomy of $\SY$, it is natural to define it using the language of groupoids. The monodromy groupoid \cite{MM} is defined as follows:

\begin{definition}
Let $(M,\SY)$ be a foliated manifold. The \emph{monodromy groupoid} $\Mon(M,\SY)$ is the set of triples $(p,q,\alpha)$ where $p$ and $q$ belong to the same leaf of $\SY$ and $\alpha$ is a homotopy class of leafwise paths connecting $p$ with $q$.

The monodromy groupoid is endowed with the following operations:
\begin{itemize}
\item[-] Source and target maps $s,t: \Mon(M,\SY) \longrightarrow M$ defined by $s(p,q,\alpha)=p$, $t(p,q,\alpha)=q$.
\item[-] A partially defined multiplication map $\Mon(M,\SY) \times_M \Mon(M,\SY) \longrightarrow M$
$$(p,q,\alpha) \times (q,q',\alpha')\longmapsto(p,q',\alpha . \alpha').$$
Here $.$ denotes concatenation of homotopy classes of paths.
\item[-] Unit map $M \longrightarrow \Mon(M,\SY)$ that takes $p$ to $(p,p,[p])$, the class of the constant path at $p$,
\item[-] Inverse map $\Mon(M,\SY) \longrightarrow \Mon(M,\SY)$ that takes $(p,q,\alpha)$ to $(q,p,\alpha^{-1})$,
\item[-] A partially defined action $\bullet: \Mon(M,\SY) \times_M M \longrightarrow M$ on $M$ defined as $(p,q,\alpha) \bullet p = q$.
\end{itemize}
\end{definition}
By construction, the orbit of a point $p \in M$ under the action is exactly the leaf $L$ of $\SY$ in which it is contained. The following result \cite{MM} states that $\Mon(M,\SY)$ can be endowed with a smooth structure:
\begin{lemma}
$\Mon(M,\SY)$ is a Lie groupoid i.e. it is a smooth manifold, possibly non--Hausdorff and non-second-countable, with smooth structure maps. Its dimension is $\dim(\SY) + \dim(M)$.\hfill$\Box$
\end{lemma}

The linear holonomy of $(p,q,\alpha)$ in $\Mon(M,\SY)$ is the identification of the normal fiber $(TM/\SY)_p$ with $(TM/\SY)_q$ provided by parallel transport along $\alpha$ using $\SY$. Globally, this translates into the action $M \acts_M \Mon(M,\SY)$ lifting to an action $\PP(TM/\SY) \acts_M \Mon(M,\SY)$ on the projective normal bundle, which is projective linear between fibres.

Let us focus on the Engel structure $\SD$. In this case $M$ is $4$--dimensional and $\SY \subset \SD$ is a line field. Over each point $p \in M$ the Engel structure determines a point $\PP(\SD/\SY)_p \subset \PP(TM/\SY)_p$. This line can be transported using the action $\PP(TM/\SY) \acts_M \Mon(M,\SY)$ described previously:
$$\gamma_\SY(\SD)_p: s^{-1}(p) \longrightarrow \PP(TM/\SY)_p,\qquad \gamma_\SY(\SD)_p(p,q,\alpha) = (q,p,\alpha^{-1}) \bullet \PP(\SD/\SY)_q.$$
Note that the domain $s^{-1}(p)$ of the curve $\gamma_\SY(\SD)_p$ is diffeomorphic to $\R$.

\begin{definition} \label{def:developmentMap}
The smooth map $\gamma_\SY(\SD): \Mon(M,\SY) \longrightarrow \PP(TM/\SY)$ obtained by glueing the collection of all maps $\gamma_\SY(\SD)_p$ is called the \textbf{development map} of $\SD$ along $\SY$.
\end{definition}
In Subsection \ref{ssec:EngelCharacterisation} we explained how the 2-plane field $\SD$ can be described as a family of curves in $\R\PP^2$. The development map provides an intrinsic description of the same phenomenon. By construction, the map $\gamma_\SY(\SD)$ is equivariant for the action $\PP(TM/\SY) \acts_M \Mon(M,\SY)$:
\[ \gamma_\SY(\SD)(p,q,\alpha) = (p',p,\alpha') \bullet \gamma_\SY(\SD)((p',p,\alpha') . (p,q,\alpha)). \]
That is, the curve $\gamma_\SY(\SD)_p \subset \PP(TM/\SY)_p$ is obtained from the curve $\gamma_\SY(\SD)_{p'} \subset \PP(TM/\SY)_{p'}$ using the linear holonomy identification between the two spaces. The first condition (A) in Proposition \ref{prop:EngelCharacterisation} implies the following:

\begin{lemma} \label{lem:EngelCharacterisation}
The module $[\SD,\SD]$ is a $3$-distribution if and only if each curve $\gamma_\SY(\SD)_p$ is an immersion. Furthermore, if the curves $\gamma_\SY(\SD)_p$ have no inflection points, $\SD$ is an Engel structure. \hfill$\Box$
\end{lemma}
In Definition \ref{def:wiggle} we defined a wiggle as an immersion of the interval into $\R\PP^2$ which closes up to a smooth, convex, embedded curve. As such, wiggles are defined in terms of self-intersections and tangencies, implying that being a wiggle is well-defined up to projective transformations of $\R\PP^2$. This allows us to speak of intervals of $\gamma_\SY(\SD)_p$ being wiggles in an intrinsic manner, without referring to any particular identification of $\PP(TM/\SY)_p$ with $\R\PP^2$. This fact will be important for our arguments.

\begin{remark}
Definition \ref{def:developmentMap} recovers the notion introduced by Montgomery in \cite{Mon}. Indeed, if $\SD$ is Engel and $\SY$ is the kernel $\SW$, its linearized holonomy preserves the planes $(\SE/\SW)_p$. In consequence, the monodromy groupoid $\Mon(M,\SY)$ acts on the projectivized bundle $\PP(\SE/\SW)$. Since $\SD \subset \SE$, the development map takes values in $\PP(\SE/\SW)$ and therefore the Engel condition (B) in Proposition \ref{prop:EngelCharacterisation} implies that each curve $\gamma_\SW(\SD)_p$ is an immersion.\hfill$\blacksquare$
\end{remark}

\section{$h$-principle for loose Engel families} \label{sec:main}

In the theory of $h$-principles \cite{EM,Gr86} there is particular value in finding the correct subclass of structures adhering to an $h$-principle \cite{BEM,CP,El89}. In the present paper, the $h$-principle is a consequence of the flexibility provided by a global dynamical property called \emph{looseness}. This notion is given in Definition \ref{def:loose}, Subsection \ref{ssec:mainDef}.

Then we prove the existence Theorem \ref{thm:main1} (Subsection \ref{ssec:existence}) and the uniqueness Theorem \ref{thm:main2} (Subsection \ref{ssec:uniqueness}). Bringing the two of them together allows us to deduce Theorem \ref{thm:main} (Subsection \ref{ssec:main}).

\subsection{Loose Engel Structures} \label{ssec:mainDef}

Lemma \ref{lem:convexity} implies that adding enough loops to an immersed curve in $\R\PP^2$ makes it convex (after a suitable modification in-between the cutting points). In line with other $h$-principles \cite{BEM,Gr86}, once the curve is convex and a loop is added, arbitrarily many new loops can be introduced while preserving convexity. These two phenomena have direct implications in the world of Engel structures.

First, given a $2$-plane distribution in a smooth $4$-manifold $M$, we can make it Engel by adding sufficiently many wiggles to its development map, proceeding carefully over a covering of $M$. Secondly, if convexity has been achieved and there are enough wiggles available, we can add arbitrarily many more while keeping the development map convex. These are the two main ingredients to prove a relative $h$--principle for this particular class of Engel structures.

The precise definition of this subclass can be detailed as follows. Let $M$ be a smooth $4$-manifold, $K$ a compact CW--complex, and $N$ a positive integer. Consider a continuous family of Engel structures $\SD: K \longrightarrow \Engel(M)$ and line fields $(\SY(k))_{k \in K}$ with $\SY(k) \subset \SD(k)$. Let $\gamma_{\SY(k)}(\SD(k))$ denote the corresponding development maps.

\begin{definition} \label{def:loose}
A family of Engel structures $\SD$ is \textbf{$N$--loose} with \textbf{certificate} $\SY$ if:
\begin{itemize}
\item[-] the development curves $\gamma_{\SY(k)}(\SD(k))_p$ are convex and
\item[-] for each $k \in K$ and $p \in M$, there is a segment $\gamma \subset \gamma_{\SY(k)}(\SD(k))_p$ containing $N$ wiggles that projects to an embedded curve $t(\gamma)$ under the target map.
\end{itemize}

The family $\SD$ is said to be \textbf{$\infty$--loose} if this holds for every positive $N$.
\end{definition}
The convexity condition for the development map implies that the line field $\SY(k)$ is always transverse to the kernel of the Engel structure $\SD(k)$. The embedding condition for the segment $\gamma$ implies that $\infty$--looseness can only hold for line fields $\SY(k)$ with no closed orbits; see Proposition \ref{prop:PPP}.

\begin{remark}
In \cite{Sal}, N. Saldanha describes the homotopy type of the space of convex curves in $\NS^2$. He shows that convex curves behave flexibly as soon as a loop is introduced. He called such families of curves \emph{loose}. We have decided to name our flexible families of Engel structures accordingly. The geometric intuition we have is that the flexibility of loose Engel structures is a manifestation of the flexibility displayed by convex curves. \hfill$\blacksquare$
\end{remark}

\begin{remark}
In \cite[Lemma 4.1]{Sal} it is proven that certain bounds on the total curvature imply that a convex curve has a wiggle (up to homotopy through convex curves). In Definition \ref{def:loose} we introduce looseness using wiggles, but one could define it instead by requiring that the development map has sufficiently large total curvature. This would in fact yield a larger class of Engel structures that would, nonetheless, be weakly homotopy equivalent to the one presented here. \hfill$\blacksquare$
\end{remark}

\subsection{Convex shells}

In our proof of Theorem \ref{thm:main1} we first upgrade M.~Gromov's Engel $h$-principle for open manifolds \cite{Gr86} to a quantitative statement. This is the content of Proposition \ref{prop:reduction}. This effectively reduces the proof to a extension problem for Engel germs in $\partial\D^4$ to the interior of $\D^4$. Following the geometric setup explained in Subsection \ref{ssec:EngelCharacterisation}, we introduce the following

\begin{definition}
A {\bf convex shell} is a $2$--distribution $\SD = \langle \partial_t,X\rangle$ in $\D^3 \times [0,1]$ such that the curves $X_p$ are immersed for all $p$ and convex at time $t$ whenever $(p,t) \in \Op(\partial(\D^3\times [0,1]))$.

In particular, $\SD$ is everywhere non--integrable and defines a germ of Engel structure along the boundary. A convex shell is said to be {\bf solid} if $\SD$ is everywhere Engel.
\end{definition}

The quantitative version reads as follows. Let $N$ be a positive integer. A convex shell is {\bf $N$--convex} if there exist:
\begin{itemize}
\item[-]a constant $\varepsilon\in (0,1)$,
\item[-]functions $(t_i: \D^3 \longrightarrow (0,\varepsilon))_{i=1,\dots,n}$ with $0<t_1(p)< \cdots < t_n(p)< \varepsilon$, and
\item[-]a $\D^3$-family of convex curves $(f_p: [0,\varepsilon] \longrightarrow \R\PP^2)_{p \in \D^3}$ such that $X_p = f_p^{[1\# t_1(p),\dots,1\# t_N(p)]}.$
\end{itemize}
The definition of parametric families of $N$--convex shells is given by the natural extension to higher-dimensional families of curves.

\subsection{Existence of loose Engel families} \label{ssec:existence}

In this subsection we solve the parametric extension problem for convex shells. We will prove the following version of Theorem \ref{thm:main1}:
\begin{proposition} \label{prop:main1}
Let $M$ be a smooth $4$-manifold, $K$ a compact CW--complex, and $N$ a positive integer. Then, any family $\SD_0: K \longrightarrow \FEngel(M)$ is formally homotopic to an $N$--loose family $\SD_1$.

Fix a family of line fields $\SY = (\SY(k))_{k \in K}$ with $\SY(k) \subset \SD_0(k)$ transverse to the kernel $\SW_0(k)$. Then, $\SD_1$ can be assumed to have $\SY$ as its certificate of $N$-looseness. Additionally, the constant $N$ can be taken to be $\infty$ if the $\SY(k)$ have no closed orbits.
\end{proposition}

Proposition \ref{prop:main1} is proven in two stages, following the structure in $h$--principles of \emph{reducing} to a standard model and then \emph{extending} the boundary germ to the interior. The first step is achieved in the following
\begin{proposition} \label{prop:reduction}
Let $M$ be a smooth $4$-manifold, $K$ a compact CW--complex, and $N$ a positive integer. Consider a family of formal Engel structures $\SD_0: K \longrightarrow \FEngel(M)$ and line fields $\SY(k) \subset \SD_0(k)$ transverse to the formal kernel $\SW_0(k)$.

Then, there exists a collection of disjoint balls $\{B_i\}_{i\in I}\subset M \times K$ and a homotopy
$$ \SD_s: K \longrightarrow \FEngel(M), \qquad s \in [0,1],$$
such that $\SY(k) \subset \SD_s(k)$ is transverse to the formal kernel $\SW_s(k)$, and
\begin{itemize}
\item[-] for each $i\in I$, $\SD_1|_{B_i}$ is a $\D^{\dim(K)}$--family of $N$--convex shells with respect to $\SY$, 
\item[-] $\SD_1(k)$ is Engel in $p \in M$ if $(p,k)\in (M\times K)\setminus \bigcup_{i\in I} B_i$.
\end{itemize}
\end{proposition}

We will prove Proposition \ref{prop:reduction} by using the following auxiliary Lemma:
\begin{lemma} \label{lem:nonIntegrable}
Let $M$ be a smooth $4$-manifold, $K$ a compact CW--complex. Consider a family of formal Engel structures $\SD_0: K \longrightarrow \FEngel(M)$ and line fields $\SY(k) \subset \SD_0(k)$ transverse to the formal kernel $\SW_0(k)$.

Then, there is a homotopy
$$ \SD_s: K \longrightarrow \FEngel(M), \qquad s \in [0,1] $$
such that $\SY(k) \subset \SD_s(k)$ is transverse to the formal kernel $\SW_s(k)$ and the formal even--contact structure $\SE_1(k)$ is given by $[\SD_1(k),\SD_1(k)]$.
\end{lemma}
\begin{proof}
Let us assume first that $K$ is a compact manifold. Consider a triangulation $\ST$ of $M \times K$. Regard the family of line fields $\SY$ as a line field in $M\times K$. Then assume that the triangulation $\ST$ is in general position \cite{Th0,Th} with respect to $\SY$ and the foliation by fibres of $M \times K \longrightarrow K$. In particular, all lower dimensional simplices are transverse to the line field $\SY$.

Now, to each simplex $\sigma$ of $\ST$ we associate a $\SY$--flowbox $\SU(\sigma)$ such that
\begin{itemize}
\item[-] the set of all such flowboxes is a covering of $M \times K$,
\item[-] two flowboxes only intersect each other if one of the simplices is contained the other,
\item[-] if $\sigma$ is top dimensional, $\SU(\sigma)$ is obtained from $\sigma$ by a $C^0$--small shrinking,
\item[-] if $\sigma$ is not top dimensional, any $\SY$--interval in $\SU(\sigma)$ is either fully contained or completely disjoint from the flowboxes corresponding to subsimplices.
\end{itemize}
These $\SY$--flowboxes are constructed in \cite[Proposition 29]{CPPP}. In short, $\SU(\sigma)$ is obtained by shrinking $\sigma$ and then thickening in the directions complementary to $\sigma$. 

The 2-distribution $\SD_0$ can be modified over each $\SU(\sigma)$ inductively in the dimension of $\sigma$, relatively to previous steps. Let us denote by $\SE_0$ the $K$-family of $3$--distributions which is part of the formal data. Note that, over each flowbox, the Engel family $\SD_0$ can be regarded as a $\D^3 \times \D^{\dim(K)}$--family of formal immersions of the interval into the projective plane. Indeed, the $2$--distribution $\SD_0$ provides a family of curves into $\R\PP^2$ and the $3$--distribution $\SE_0$ provides a great circle at each point of the curves. The isomorphism $\det(\SD_0) \equiv \SE_0/\SD_0$ encoded in the formal data -- Equation (\ref{eq:iso1}) -- provides an orientation\footnote{The isomorphism $\det(\SD_0) \equiv \SE_0/\SD_0$ tells us that $\SE_0$ is canonically oriented globally. In each flowbox we make a choice of orientation for $\SY$, which automatically orients $\SE_0/\SY$. Its projectivisation is the great circle under consideration, which inherits an orientation.} of each great circle. Then, the relative nature of the Smale-Hirsch theorem \cite{EM,Gr86} implies that we can modify the curves so that they become immersions, relative to previous flowboxes. In terms of the formal Engel structure this means that $\SD_0$ is formally homotopic to a family of non--integrable plane fields that bracket-generate a $3$--distribution homotopic to $\SE_0$.

This proves the claim. For $K$ an arbitrary CW-complex, we proceed cell by cell as just explained, using again the fact that the Smale-Hirsch theorem is relative both in parameter and domain.
\end{proof}

\begin{proof}[Proof of Proposition \ref{prop:reduction}]
We use the setup explained in the proof of Lemma \ref{lem:nonIntegrable}: We may assume that $K$ is a compact manifold. We fix a triangulation $\ST$ of $M \times K$ in general position with respect to $\SY$ and the fibres of $M \times K \longrightarrow K$. This allows us to cover $M\times K$ by $\SY$--flowboxes. By the Lemma, we can assume that $[\SD_0,\SD_0]$ is the $3$--distribution $\SE_0$ given by the formal data.
 
Now we modify the 2-distribution $\SD_0$ over each flowbox $\SU(\sigma)$, inductively in the dimension of $\sigma$ for $\dim(\sigma) < \dim(M \times K)$. We regard the restriction $\SD_0|_{\SU(\sigma)}$ to each flowbox as a $\D^3 \times \D^{\dim(K)}$--family of immersions $X_{p,k}$ of the interval $I$ into $\R\PP^2$. The isomorphism $\det(\SE_0/\SW_0) \equiv TM/\SE_0$ provided by the formal data -- Equation (\ref{eq:iso2}) -- provides a local orientation\footnote{The isomorphism $\det(\SE_0/\SW_0) \equiv TM/\SE_0$ provides a canonical orientation for the bundle $TM/\SW_0$. In the flowbox we choose an auxiliary orientation for $\SY$. Since $\SY$ is contained in $\SD_0$ and transverse to $\SW_0$, we obtain an orientation of $TM/\SD_0$. This yields the local orientation of $\R\PP^2$.} of $\R\PP^2$; we want the curves $X_{p,k}$ to be convex with respect to this orientation.

In line with Proposition \ref{prop:slidingWiggles}, we first use Lemma \ref{lem:convexity} to add arbitrarily many wiggles to each $X_{p,k}$ close to the endpoints $\partial I$ and then we distribute them evenly in the interior $I \setminus \Op(\partial I)$. This is done parametrically in the band $\{1-\varepsilon \leq |p|,|k| \leq 1\}$, with $\varepsilon>0$ arbitrarily small. Hence, we can assume that $X_{p,k}$ is convex and has arbitrarily many wiggles away from its endpoints if $\{|p|,|k| \leq 1-\varepsilon\}$. Note that the behaviour of $X_{p,k}$ will be quite complicated close to $\partial I$.

This process is relative to the boundary of the flowbox and it can also be made relative to previous flowboxes: By assumption, the development map of the 2-distribution $\SD_0$ along a $\SY$--curve $X_{p,k}$ contained in a previous flowbox is already convex. Since the development map is intrinsically defined, we have a precise control of how many wiggles such a $X_{p,k}$ has. In particular, it can be assumed to be arbitrarily large by evenly introducing sufficiently many wiggles in the previous steps. Proposition \ref{prop:slidingWiggles} can be applied relative to the set of these curves. 

The argument can now be repeated until we reach the top dimensional cells. The collection of balls $\{B_i\}_{i\in I}$ in the statement of Proposition \ref{prop:reduction} is taken to be the collection of flowboxes $\SU(\sigma) \subset \sigma$ with $\sigma$ top dimensional. Since we have added arbitrarily many wiggles along the codimension-1 skeleton, the formal Engel structure is a genuine Engel structure in the boundary of each ball $B_i$, for all $i\in I$. In addition, each ball $B_i$ is a $\D^{\dim(K)}$--family of $N$--convex shells as required for the statement. Finally, observe that $\SY$ has remained fixed during this formal homotopy, which concludes the proof.
\end{proof}
Proposition \ref{prop:reduction} solves the reduction process for Proposition \ref{prop:main1}. Let us now address the extension problem.

Consider the $\D^{\dim(K)}$--families of shells $\{B_i\}_{i\in I}$ produced by Proposition \ref{prop:reduction}. Observe that the restriction of the 2-distribution $\SD_1|_{B_i}$ can be regarded as a $\D^3 \times \D^{\dim(k)}$--family of curves $X_{p,k}$ satisfying the hypothesis of Proposition \ref{prop:slidingWiggles}. Here $\D^3 \times \D^{\dim(k)}$ plays the role of $K$ and $A$ is its boundary. From this we deduce that there is a deformation, relative to the boundary of the model, that makes all curves convex. This implies that there is an Engel family $\SD_2$ that is formally homotopic to $\SD_1$. Additionally, $\SD_2|_{B_i}$ is a $(N-1)$--convex shell, since we only needed to use one of the wiggles during the homotopy. This argument proves the following
\begin{proposition} \label{prop:extension}
Let $K$ be a compact CW-complex. Any family $(\D^3 \times [0,1],\SD_k)_{k \in K}$ of $N$--convex shells is homotopic to a family of solid $(N-1)$--convex shells. This is relative to the boundary of the shells, and relative to the parameter region in which they are already solid.\hfill$\Box$
\end{proposition}

\begin{proof}[Proof of Proposition \ref{prop:main1} and Theorem \ref{thm:main1}]
Consider the shells $B_i$ constructed in Proposition \ref{prop:reduction}. Since these are obtained by shrinking a top-dimensional simplex of the triangulation $\ST$ of $M \times K$, every orbit of $\SY$ must intersect some ball $B_i$ in the collection. An application of Proposition \ref{prop:extension} turns each $B_i$ into a solid $(N-1)$--convex shell, and therefore $\SD_2$ is $(N-1)$--loose. This proves Theorem \ref{thm:main1}.

Assume now that $\SY$ has no closed orbits. Then every orbit accumulates somewhere and therefore intersects one of the $B_i$ infinitely many times. Since wiggles are intrinsically defined using the development map, we deduce that each orbit of $\SY$ has infinitely many of them and therefore $\SD_1$ is $\infty$--loose.
\end{proof}

This concludes the existence $h$-principle for the class of $\infty$-loose Engel structures, as stated in Theorem \ref{thm:main}. In particular, we have an existence $h$-principle refining our previous result \cite{CPPP}, which we will now further improve to a uniqueness $h$-principle.

\subsection{Uniqueness of loose Engel families} \label{ssec:uniqueness}

In this subsection we show that the $N$-loose Engel families constructed in Theorem \ref{thm:main1} are unique up to homotopy if $N$ is large enough. This is precisely the content of Theorem \ref{thm:main2}; its quantitative nature is in line with other quantitative phenomena appearing in higher-dimensional contact flexibility \cite{BEM,CMP}. We have included a discussion on this in Section \ref{sec:app}.

Theorem \ref{thm:main2} will be proven by first showing that any loose family can be homotoped to resemble a family produced by Theorem \ref{thm:main1}. This is the content of the following
\begin{proposition} \label{prop:bootstrapping}
Let $M$ be a smooth $4$-manifold and $K$ a compact CW--complex. There exists a positive integer $N_0$, depending only on the dimension of $K$, such that for any:
\begin{itemize}
\item[-] $N$--loose family $\SD_0: K \longrightarrow \Engel(M)$, $N\geq N_0$, with certificate $\SY$,
\item[-] triangulation $\ST$ of $M \times K$ in general position with respect to $\SY$ and $M \times K \to K$,
\item[-] covering $\{\SU(\sigma)\}_{\sigma \in \ST}$ as in Proposition \ref{prop:reduction},
\item[-] non-negative integer $N_1$,
\end{itemize}
there is a homotopy $\SD_s: K \longrightarrow \Engel(M)$ satisfying
\begin{itemize}
\item[-] $\SY(k) \subset \SD_s(k)$ is transverse to the kernel $\SW_s(k)$,
\item[-] $\SD_s$ is $(N-N_0)$--loose for all $s$, with $\SY$ as its certificate of looseness,
\item[-] for any any top--dimensional simplex $\sigma \in \ST$, $\SD_1|_{\SU(\sigma)}$ is a family of solid $N_1$--convex shells.
\end{itemize}
\end{proposition}

We will say that a family of Engel structures is simply \textbf{loose} if it is $N_0$--loose. During the proof we will provide a bound for the constant $N_0$. 

\begin{proof}
Since the Engel structure $\SD_0$ is $N$--loose, at any point $(p,k) \in M \times K$ we can find an embedded interval $\gamma \subset M \times \{k\}$ tangent to $\SY$, containing $(p,k)$, and whose development map has $N$ wiggles. By thickening such intervals, we find a covering $\{\SU_i\}$ of $M \times K$ by solid $N$--convex shells. It is sufficient for us to show that there is an Engel homotopy $(\SD_s)_{s \in [0,1]}$ through $(N-N_0)$--loose Engel structures such that the development map of $\SD_1$ has arbitrarily many uniformly distributed wiggles. This can be achieved by modifying the development map inductively over each element $\SU_i$ of the covering, as follows.

Start with the first shell $\SU_0$, where $\SD_0$ is considered as a family of convex intervals $(X_{p,k})_{(p,k)\in\D^3 \times \D^{\dim(k)}}$, and fix $\varepsilon>0$ arbitrarily small. Since we have $N$ wiggles, we can apply Proposition \ref{prop:slidingWiggles} to one of them to produce arbitrarily many more wiggles in the region $\{|p|,|k| \leq 1-\varepsilon\}$. These wiggles can be assumed to be uniformly distributed in the domain. Note that in doing this, the wiggle we chose in the region $\{1-\varepsilon \leq |p|\} \cup \{1-\varepsilon \leq |k|\}$ disappears as Little's homotopy is performed. In particular, $\SU_0$ is only a $(N-1)$--convex shell for the new Engel structure.

Consider now $\SU_1$ and suppose that it intersects $\SU_0$. From the perspective of $\SU_1$, the homotopy in $\SU_0$ could have destroyed two wiggles. Indeed, the wiggle we used for the homotopy in $\SU_0$ may intersect at most two wiggles in $\SU_1$ (see Remark \ref{rem:wiggle}). However, if we assume that $N>2$, there is at least one wiggle remaining and we can repeat the argument above. This allows us to arbitrarily increase the number of wiggles in the interior of $\SU_1$. It is natural to proceed by repeating this process inductively over the covering index $i$. In order to do that, denote the projection to the orbit space by $\pi:\SU_i \longrightarrow U_i= \SU_i/\SY$, where each $U_i$ is diffeomorphic to $\D^3 \times \D^{\dim(K)}$.

The main geometric ingredient in the proof is showing that the covering $\{\SU_i\}$ can be chosen such that:
\begin{enumerate}
\item[I.]  Only $(N_0-1)$ of the wiggles of a given shell $\SU_i$ get destroyed by previous homotopies.
\item[II.] There exists a continuous section $U_i \longrightarrow \SU_i$ that provides a marked wiggle in each flowline.
\end{enumerate}

For that, fix a cover $\{\SV_i\}$ using the process described in the first paragraph, and write
$$\pi_i: \SV_i \longrightarrow V_i = \SV_i/\SY$$
for the canonical projection. The intersection $\pi_i(\partial\SV_{i'} \bigcap \SV_i)$ defines a codimension--$1$ submanifold $S_{i'} \subset V_i$, and by a small perturbation of the flowboxes $\{\SV_{i'}\}$, we can assume that the submanifolds $\{S_{i'}\}_{i' \neq i}$ of $V_i$ intersect transversely. In particular, a point in $V_i$ may only lie in $C=\dim(K)+4$ different manifolds $S_{i'}$. In the previous steps ($i' < i$) of the induction, Little's homotopy destroyed two wiggles in each region $\Op(\partial\SV_{i'})$. Hence, by setting $N_0 \geq 2C+1$ it follows that each $\SY|_{\SV_i}$ flowline contains still one wiggle, so Condition (I) holds. 

In order to show that the wiggles can be chosen in a continuous way (Condition (II)), we construct a finer covering $\{\SU_i^j\}$ of $M \times K$. This is done inductively on $i$ as follows: First set $\SU_0^0 = \SV_0$ and fix some continuous choice of wiggle $V_0 \longrightarrow \SV_0$. Suppose that we have already subdivided all $\SV_{i'}$ with $i' < i$, yielding some partial covering $\{\SU_{i'}^j\}_{i' < i}$ with corresponding choices of wiggles $\{\pi_{i'}(\SU_{i'}^j) \longrightarrow \SU_{i'}^j\}$. Now choose a very fine triangulation $\ST_i$ of $V_i$ and fix small contractible open neighborhoods $\{U_i^j\}$ of each simplex in $\ST_i$. Then $\{\SU_i^j = \pi_i^{-1}(U_i^j)\}$ is a covering of $\SV_i$ by flowboxes and we claim that this is enough to conclude.

Indeed, using transversality as above we can assume that each point in $\pi_i(\SU_i^j) \subset V_i$ meets at most $C$ of the manifolds $\{\pi_i(\partial \SU_{i'}^{j'} \cap \SU_i^j)\}_{i' < i}$. Additionally, there is a constant $D$, depending only on $\dim(K)$, bounding from above the number of simplices of $\ST_i$ intersecting a given simplex. Therefore, $\SU_i^j$ intersects at most $D$ of the flowboxes $\{\SU_i^{j'}\}_{j' \neq j}$. Condition (I) then holds by setting $N_0 \geq 2C+2D+1$. Additionally, if $\ST_i$ is fine enough, each element $U_i^j$ is a neighbourhood of a point so it is possible to make a continuous choice of wiggle. 
\end{proof}

We will need one more ingredient before we prove Theorem \ref{thm:main2}:
\begin{proposition}[\cite{PPP}] \label{prop:PPP}
Denote by $\SX(M)$ the space of line fields on a manifold $M$. Denote by $\SX_\no(M) \subset \SX(M)$ the subset of line fields without periodic orbits. The inclusion $\SX_\no(M) \subset \SX(M)$ induces a weak homotopy equivalence provided that $\dim(M) \geq 3$.
\end{proposition}
This result relies on the existence of parametric versions of the plugs of Wilson and Kuperberg. In particular, it states that the choice of a line field without periodic orbits in the statement of Theorem \ref{thm:main} is not a restriction from a homotopical point of view.

\begin{proof}[Proof of Theorem \ref{thm:main2}]
Let $\SY_i \subset \SD_i$ be the certificate of $N$--looseness of $\SD_i$, and fix
$$ \tilde\SD: K \times [0,1] \longrightarrow \FEngel(M) $$
a family of formal Engel structures connecting $\SD_0$ and $\SD_1$. Write $\tilde\SW$ for the formal kernel of $\tilde\SD$. Fix a family of line fields $\SY \subset \tilde\SD$ connecting $\SY_0$ and $\SY_1$ and transverse to $\tilde\SW$. Construct a triangulation $\ST$ of $M \times K \times [0,1]$ in general position with $\SY$, in general position with $M \times K \times [0,1] \longrightarrow K \times [0,1]$, and restricting to triangulations $\ST_i$ on $M \times K \times \{i\}$ also in general position \cite{CPPP,Th0}.

Then apply Proposition \ref{prop:bootstrapping} to achieve that for any top dimensional simplex $\sigma \in \ST_i$, the restriction $\SD_i|_{\SU(\sigma)}$ is a family of solid $N_1$--convex shells. This allows us to apply the reduction in Proposition \ref{prop:reduction} relative to the ends $M \times K \times \{0,1\}$. Then Proposition \ref{prop:extension} can be used to achieve the Engel condition in the interior of the top cells. Following the proof of Theorem \ref{thm:main1}, we have deformed the 2-distribution $\tilde\SD$ to a family of Engel structures $\SD: K \times I \longrightarrow \Engel(M)$. By construction, the 2-distribution $\SD$ restricts to $\SD_i$ in $M \times K \times \{i\}$ as desired. Finally, if $\SY_0$ and $\SY_1$ have no periodic orbits, the same can be assumed about $\SY$ after an application of Proposition \ref{prop:PPP}. Therefore, if the Engel structure $\SD_i$ are $\infty$--loose, the same holds for the homotopy $\SD$.
\end{proof}

\subsection{Proof of Theorem \ref{thm:main}} \label{ssec:main}

Consider a family $\SD: (\D^k,\partial\D^k) \longrightarrow (\FEngel(M,\SY),\Loose(M,\SY))$. Theorem \ref{thm:main1} implies that $\SD(0)$ can be homotoped to be $\infty$--loose. Then we can regard $\SD$ as a formal homotopy between the family $\SD|_{\partial\D^k}$ and the constant family $\SD(0)$. Applying Theorem \ref{thm:main2} shows that $\SD$ can be homotoped to have image in $\Loose(M,\SY)$, as desired. \hfill$\Box$

This concludes the $h$-principle for loose Engel structures.

\section{Applications}\label{sec:app}

In this section we prove Corollary \ref{cor:prolong} on Engel prolongations, and discuss two additional families of examples of loose Engel structures. It follows from Theorem \ref{thm:main2} that these families satisfy the $h$-principle, and thus exhibit completely flexible behaviour.

\subsection{Prolongations}\label{ssec:EngelProlong}

\'E.~Cartan introduced in \cite{Car} the notion of \emph{prolongation} for contact structures and Lorentzian metrics, which we exploited in recent work \cite{CPPP} to manipulate Engel structures locally. Let us review these two constructions and prove Corollary \ref{cor:prolong}.

Let $V$ be a smooth oriented $3$-manifold and $\xi$ an oriented $2$--plane distribution. By definition, the associated oriented \textbf{formal Cartan prolongation} $(M(\xi),\SD(\xi))$ is the sphere bundle $M(\xi):= \NS(\xi) \stackrel{\pi}{\longrightarrow} V$ endowed with the tautological distribution
\begin{equation} \label{eq:taut}
\SD(\xi)(p,l) = \pi^* [l].
\end{equation}
The distribution is Engel if and only if $\xi$ is a contact distribution. Indeed, this is Condition (B) in Proposition \ref{prop:EngelCharacterisation}. In this case, the Engel structure $(M(\xi),\SD(\xi))$ is called the \textbf{Cartan prolongation} of the contact structure $\xi$.
 
Suppose instead that the $3$-manifold manifold $V$ is endowed with a Lorentzian metric $g$. The kernel $C_g$ of the Lorentzian metric, known as the {\it light cone}, defines a sphere bundle $M(g) := \PP(C_g) \stackrel{\pi}{\longrightarrow} V$ endowed with a tautological distribution $\SD(g)$ defined again by Equation (\ref{eq:taut}). This distribution is always an Engel structure, since it satisfies Condition (A) in Proposition \ref{prop:EngelCharacterisation}. By definition, the Engel structure $(M(g),\SD(g))$ is the \textbf{Lorentz prolongation} of $g$.

A particular case which is of interest for us is as follows: Consider an orientable and coorientable plane field $\xi$. Endow it with a metric $g_\xi$ and pick a complementary vector field $\nu$. Then, the pair $(g_\xi,\nu)$ defines a family of Lorentz metrics $(g_r)_{r \in \R^+}$ by declaring the vector field $\nu$ to be orthogonal to $\xi$ and of norm $-r$. As $r$ goes to infinity, the light cone $C_{g_r}$ converges to the plane field $\xi$. We can apply the prolongation construction parametrically in $r$. This allows us to think of the structures $(M(g_r), \SD(g_r))$ as convex push-offs \cite{dP17} of the formal Cartan prolongation $(M(\xi),\SD(\xi))$.

Conversely, given any Lorentz structure $g$ and any space-like plane field $\xi$, there exists a unique line complement such that $g$ is a push-off of a metric in $\xi$ by the recipe above.

Corollary \ref{cor:prolong} states that these Engel structures are all loose. The key ingredient is the observation that the light-cone intersects the unit sphere in a convex curve which is embedded (i.e. a wiggle):
\begin{proof}[Proof of Corollary \ref{cor:prolong}]
Let $K$ be a compact parameter space, $\SD$ a $K$--family of Lorentz prolongations, and $\xi$ a $K$--family of plane fields in $V$ such that $\SD(k)$ is a convex push-off of $\SD(\xi(k))$ given by a family of functions $r(k)$ (and a corresponding family of metrics in the plane fields $\xi(k)$). By performing an Engel homotopy, given by increasing the real numbers $r(k)$, we can assume that $d\pi(\SD(k))$ is arbitrarily close to $\nu(k)$, where $\nu(k): V \longrightarrow TV$, $k \in K$, is a vector field transverse to $\xi(k)$.

Now consider a family of line fields $\SY_s(k) \subset \SD(k)$, $s \in [0,1]$, spanned by vector fields $Y_s(k)$, with $\SY_0(k)$ contained in the fibre direction and all others transverse to it. Over any $3$--disc in $V$ (lifted to the fibre bundle by taking a section), the vector fields $Y_s(k)$ provide a family of return maps $\phi_{k,s}$, with $\phi_{k,0}$ the identity.

We claim that, for any fixed $N \in \N$, the iterates of the return maps
$$\phi^{(j)}_{k,s},\quad j=1, \ldots, N,$$
have no fixed points if $s\neq 0$ is close enough to $0$. Indeed, since we have pushed the prolongations to be very convex, $\phi^{(j)}_{k,s}$ becomes a map that approximates an arbitrarily short time flow of $\nu(k)$. By compactness of $V$, the map cannot have fixed points. Now the claim follows by taking $N$ larger than the universal constant $N_0$ corresponding to a family of dimension $\dim(K)$: the resulting family of Engel structures is $N$--loose. This proves the claim in the Lorentz case. Suppose now that $\SD$ is family of orientable Cartan prolongations, then $\SD$ is homotopic to a family of Lorentz prolongations by a convex push--off, which proves the statement as desired.
\end{proof}

A striking consequence of Corollary \ref{cor:prolong} is the following: Given two non-isotopic contact structures homotopic as plane fields, their Cartan prolongations are Engel homotopic. In particular, any non--formal contact invariant becomes formal by taking Cartan prolongations.

\begin{remark}
In the literature, see for instance \cite{dP17}, a more general notion of Cartan prolongation is considered: Assume, imposing the obvious condition on the Euler class, that there is a sphere bundle $E$ that $m\colon\!1$ covers $\NS(\xi)$. We can then define Engel structures on $E$ by pulling-back $\SD(\xi)$; similarly, we can construct $m\colon\!1$ coverings of Lorentzian prolongations. The general statement is then: Given $\SD: K \to \Engel(M)$ and $\pi: \hat{M} \to M$ a $m\colon\!1$ cover, we can construct a family $\pi^*\SD: K \to \Engel(M)$ by pull-back. If $\SD$ is $N$--loose, then $\pi^*\SD$ is $mN$--loose. \hfill$\blacksquare$
\end{remark}

\subsection{Other loose families in the literature} \label{ssec:cppp}

In this subsection we show that the families of Engel structures constructed in \cite{CPPP} and \cite{CPV} are loose. This is shown for the former class in the following
\begin{proposition}
Let $M$ be a closed $4$-manifold and let $K$ be a compact manifold. Any family of Engel structures $\SD: K \longrightarrow \Engel(M)$ constructed using the $h$-principle in \cite{CPPP} is loose up to Engel homotopy.
\end{proposition}
\begin{proof}
The construction in \cite{CPPP} produces a family $\SD: K \longrightarrow \Engel(M)$ of Engel structures with corresponding line fields $\SY: K \longrightarrow \SX(M)$, $\SY(k) \subset \SD(k)$, such that the associated development maps $\gamma_{\SY(k)}(\SD(k))$ satisfy:
\begin{enumerate}
\item[-] The curves $\gamma_{\SY(k)}(\SD(k)_p$ are immersed and weakly convex,
\item[-] There is a finite number of disjoint $3$--weakly convex shells $\SU_i$ that together cover the orbit space $(M \times K)/\SY$.
\end{enumerate}
A curve in $\R\PP^2$ is said to be weakly convex if its curvature is greater or equal than zero. Weakly convex wiggles and $N$-weakly convex shells are defined in the natural manner. The balls $\SU_i$ are of the form introduced in Proposition \ref{prop:reduction}: they are flowboxes obtained by shrinking the top cells of a triangulation of $M \times K$ in general position with respect to $\SY$.

It follows from Proposition \ref{prop:EngelCharacterisation} that $\SY(k)_p = \SW(k)_p$ if and only if the development map has an inflection point at $p$. This implies that the tangencies, defined by $\SY(k) = \SW(k)$, are all degenerate and there is a $C^{\infty}$--perturbation $\SY' \subset \SD$ of $\SY$ that is everywhere transverse to $\SW$. The development map of $\SD$ along $\SY'$ is everywhere convex. We can then perturb the collection $\{\SU_i\}$ to a collection of $\SY'$--flowboxes $\{\SU_i'\}$ covering the orbit space $(M \times K)/\SY'$ such that $\SU_i'$ is a $3$--convex shell. By applying Proposition \ref{prop:slidingWiggles} we can deform the Engel structure on each $\SU_i'$ so that it becomes a solid $N$--convex shell, with $N$ arbitrarily large. This produces a new family $\SD': K\longrightarrow \Engel(M)$. Since each orbit of $\SY'$ intersects at least one $U_i$, we obtain that the family is loose with $\SY'$ its certificate of looseness.
\end{proof}

\begin{remark}
The following is a technical observation. The Engel structures constructed in \cite{CPPP} depend on a real parameter $E >0$ that needs to be chosen large enough. There is  also an increasing function $N: \R^+ \longrightarrow \Z^+$, such that $\lim_{E \longrightarrow \infty} N(E)=\infty$. Now, the families of Engel structures $\SD: K \longrightarrow \Engel(M)$ constructed using the $h$-principle in \cite{CPPP} satisfy that the open balls $\SU_i'$ are $N(E)$-convex shells. Hence, for $E$ large enough, the original family is already loose, without having to deform it. \hfill$\blacksquare$
\end{remark}

The recent article \cite{CPV} constructs Engel structures adapted to open books, in line with the contact Giroux correspondence. Away from the binding, which is a disjoint union of tori, the structures can be understood as Cartan prolongations of a contact manifold with trivial Euler class. The following statement is proven in \cite{CPV}:
\begin{proposition}
The Engel structures constructed in \cite{CPV} are loose.
\end{proposition}
\begin{proof}
The construction in \cite{CPV} depends on a constant $k \in \Z^+$ which measures the number of turns performed by the Engel structure in terms of a legendrian framing on the page. We denote by $\SD_k$ the Engel structure that turns $k$ times. In order to extend it to the binding we need a canonical model on it that requires $k$ to be odd. Little's homotopy implies that $\SD_k$ and $\SD_{k+2}$ are homotopic. For $k$ large enough the structure is loose, since the number $k$ precisely accounts for the turning of $\SD_k$ in terms of the development map.
\end{proof}

\section{Appendix: flexibility in Engel and contact topology} \label{sec:appendix}

In this appendix we discuss the interaction between Engel structures \cite{CP,CPPP}, contact structures \cite{El89,Ge}, and the $h$-principle \cite{EM,Gr86}. Its goal is to study the manifestations and subtleties of the $h$-principle as seen from the recent new perspectives \cite{BEM,CPPP,PV,Vo}. Let us start with contact structures as the prism through which we are used to looking at the $h$-principle.

\subsection{Contact flexibility}

Even though contact structures do not abide by the $h$--principle \cite{Ge}, there is a subset of \emph{overtwisted} contact structures whose behaviour is flexible, i.e.~ their classification up to homotopy is governed by the underlying formal data. This display of flexibility is precise at the $\pi_0$-level, but for higher homotopy groups the picture is more subtle, as we explain.

Let $N$ be a closed orientable $(2n+1)$-manifold, $\FCont(N,\Delta)$ the space of almost contact structures with overtwisted disc $\Delta$ \cite{Gr86} and $\Cont_\OT(N,\Delta)$ the subspace of contact structures also overtwisted with disc $\Delta$. The main result in \cite{BEM,El89} is that the forgetful inclusion
\[ \Cont_\OT(N,\Delta) \longrightarrow \FCont(N,\Delta) \]
is a weak homotopy equivalence. This is where the first subtlety arises: the overtwisted disk $\Delta$ has been fixed. Recently, it has been shown that the space of overtwisted contact structures does not have, necessarily, the homotopy type of the space of formal contact structures \cite{Vo}. This failure of flexibility is precisely related to the homotopy type of the space of overtwisted discs in a fixed contact structure.

The articles \cite{BEM,El89} actually prove a stronger result, in which the overtwisted disc is allowed to vary: Let $\xi_0,\xi_1: K \longrightarrow \Cont(N)$ be two $K$--families of contact structures, with $K$ a compact CW--complex. Let $\Delta_0,\Delta_1$ be corresponding $K$--families of overtwisted discs and assume that there is a homotopy of pairs $(\xi_t,\Delta_t)$ with $\xi_t: K \longrightarrow \FCont(N)$ having $\Delta_t$ as overtwisted discs. Then, the families $\xi_0$ and $\xi_1$ are homotopic through contact structures relative to $\Delta_t$. Conversely, if $\xi_t: K \longrightarrow \Cont(N)$ is a homotopy between $\xi_0$ and $\xi_1$, and $\xi_0$ admits a family of overtwisted discs $\Delta_0$, we deduce from Gray stability that $\xi_t$ lifts to a homotopy of pairs $(\xi_t,\Delta_t)$.

That is, the $K$-family $\xi_0$ presents a flexible behaviour if a choice of $\Delta_0$ exists. This leads us to introduce the following definition, formalizing an ubiquitous idea in the theory of $h$-principles \cite{EM}:
\begin{definition}
Let $\xi_0$ be a $K$-family of contact structures. A continuous choice of $\Delta_0$ is said to be a \emph{certificate of overtwistedness} for the \emph{overtwisted family} $\xi_0$.
\end{definition}

\subsubsection{Overtwisted classes}

The $h$--principle in contact geometry does not hold without the mediation of a certificate, and the central obstruction is its homotopy type. At the most basic level, the family $\xi_0$ might not even admit a continuous choice of certificate $\Delta_0$, even if all the structures are individually overtwisted. This is known to happen \cite{Vo}: there exists a formally contractible loop of overtwisted contact structures in $\NS^3$ that admits no certificate and therefore is not contractible geometrically.

Two overtwisted families of contact structures may be formally homotopic but have certificates in different homotopy classes. However, there is a stable range in which this obstruction vanishes and an \emph{algebraic} form of the $h$--principle holds: Recall the forgetful inclusion $i: \Cont(N) \longrightarrow \FCont(N)$, and fix an overtwisted basepoint $\xi \in \Cont(N) \subset \FCont(N)$. We consider the homotopy groups $\pi_k(\Cont(N))$ and $\pi_k(\FCont(N))$ based at $\xi$. A class $\alpha \in \pi_k(\Cont(N))$ is said to be \emph{overtwisted} if it can be represented by an overtwisted family.

\begin{proposition} \label{prop:algebraicHPrinciple}
Let $N$ be a closed $(2n+1)$-manifold.  Consider the subgroup $\OT_k(N) \subset \pi_k(\Cont(N))$ consisting of overtwisted classes, for $0\leq k\leq 2n$.

Then, the inclusion $\pi_k(i): \OT_k(N) \longrightarrow \pi_k(\FCont(N))$ is a group isomorphism.
\end{proposition}
\begin{proof}
Let $\xi: \NS^k \longrightarrow \Cont(N)$ be an overtwisted family of contact structures with certificate $\Delta$. Since $k<2n+1$, after an isotopy we may assume that there is a point $p\in N$ which is not contained in any of the overtwisted discs $\Delta(a)$, $a \in \NS^k$. This allows us to use the $h$--principle to introduce an overtwisted disc at $p$, for all $\xi(a)$. Even if they are all based at the same point, the family of overtwisted discs might be non-trivial, but this non-triviality is carried by the formal type of the family $\xi$ encoded by the value of the distribution $\xi(a)$ at the point $p$. Any formal homotopy between overtwisted families having overtwisted discs centered at a fixed point lifts to a homotopy of pairs, concluding the proof.
\end{proof}
We say that a class not belonging to the \emph{overtwisted subgroup} $\OT_k(N)$ is a {\em tight} class. T.~Vogel's loop of overtwisted contact structures \cite{Vo} is the first instance of a $1$--dimensional tight family of individually overtwisted contact structures.

\subsubsection{Tight classes}

One can also observe that the tight classes have a natural group structure. First, we claim that $\Tight_k(N)= \pi_k(\Cont(N))/\OT_k(N)$ is a group, $k>0$. For this to hold, we must show that $\OT_k(N)$ is a normal subgroup. If $2\leq k \leq 2n$, this is true since the groups are abelian. For $k=1$, we have the following sequence of inclusions:
$$ \OT_1(N) \to \pi_1(\Cont(N)) \stackrel{\pi_1(i)}{\to} \pi_1(\FCont(N)) \simeq \OT_1(N). $$
And therefore $\Tight_1(N)$ is the kernel of the map $\pi_1(i)$. Then, we may interpret the quotient $\Tight_k(N)$ as a subgroup of $\pi_k(\Cont(N))$: it corresponds precisely to the homotopy classes of contact spheres that are homotopically trivial as almost contact spheres. Left multiplication with the overtwisted representative identifies the fibers over any other formal class, and therefore all the fibers of the map $\pi_k(i)$ are conjugated subgroups.

This stands in sharp contrast with the case $k=0$: The projection map $\pi_0(\Cont(\NS^3)) \to \OT_0(\NS^3)$ has one element in each fiber except for the fibre containing the standard contact structure, which contains two \cite{El89, El92}.

\subsection{Engel flexibility}

We can now look at the same concepts from the lens of Engel topology.

\subsubsection{Local and global}

Engel looseness differs from contact overtwistedness in that the definition of certificate we have given is not \emph{local}. The contact overtwisted disc is a particular model in a ball (or a particular contact germ over a $2n$--disc). In contrast, Engel looseness must be checked \emph{globally} on the manifold $M$ using the line field $\SY$. 

In \cite{PV} a \emph{local} Engel overtwisted disc is defined. It allows to prove flexibility in a manner that is analogous to the contact case. The main result there reads: let $\Engel_\OT(M,\Delta)$ be the space of Engel structures on $M$ having $\Delta$ as a (local) overtwisted disc. Let $\FEngel(M,\Delta)$ be the corresponding formal space. Then, the inclusion $\Engel_\OT(M,\Delta) \to \FEngel(M,\Delta)$ is a weak homotopy equivalence. Statements where the overtwisted disc is allowed to move parametrically also hold and overtwisted homotopy subgroups can be defined as well.

This leads to a surprising situation. On the one hand, Engel flexibility holds once a particular local model is found in the manifold; this is a consequence of the fact that the overtwisted disc appears to be the necessary ingredient to solve the Engel extension problem for any germ on $\partial\D^4$. On the other hand, families that seemingly do not possess this local model might still behave flexibly if they ``turn sufficiently with respect to some line field'', i.e. they are loose.

We then observe that looseness cannot yield an $h$--principle \emph{relative in the domain}, as overtwistedness does. The reason behind this is that the reduction process (achieving enough convexity in the codimension--$1$ skeleton, Proposition \ref{prop:reduction}) cannot be completed when the Engel structure is already fixed in some part of the domain (possibly having very little convexity). In particular, the extension problem of a germ in $\partial\D^4$ to the interior cannot be solved in full generality using looseness. 

Using the relative nature of the $h$--principle, one can show that an overtwisted Engel structure contains all possible local models up to Engel homotopy. From this one can deduce that any two definitions of local overtwistedness are equivalent. However, since looseness is a global property, it cannot be compared to overtwistedness. In particular, looseness has no known analogue in contact topology.

\subsubsection{Loose classes}

Recall the forgetful inclusion $\Engel(M) \longrightarrow \FEngel(M)$ and fix a loose basepoint $\SD \in \Engel(M)$. We may look at the groups $\pi_k(\Engel(M))$ and $\pi_k(\FEngel(M))$ based at $\SD$. A class $\alpha \in \pi_k(\Engel(M))$ is \emph{loose} if it can be represented by a loose family. Note that conjugating by a loose loop takes loose classes to loose classes.

In the $h$-principle for loose Engel structures, the homotopy type of the certificate is encoded in the formal type (since the certificate is always transverse to the kernel of the Engel structure). From this, we deduce the $h$--principle in its algebraic form for all homotopy groups (and not just in some stable range). This is yet another significant difference between {\it loose} Engel structures and {\it overtwisted} contact structures; see Proposition \ref{prop:algebraicHPrinciple}.
\begin{corollary} \label{cor:looseSubgroups}
Given a closed $4$-manifold $M$, let $\Loose_k(M) \subset \pi_k(\Engel(M))$ be the subgroup of loose classes. Then $\Loose_k(M) \longrightarrow \pi_k(\FEngel(M))$ is a group isomorphism. \hfill$\Box$
\end{corollary}
Similarly, in \cite{PV} it is shown that overtwisted Engel families yield subgroups $\OT_k(M) \subset \pi_k(\Engel(M))$ in the range $0 \leq k \leq 3$ (where the basepoint is instead taken to be overtwisted). We may then speak of tight classes: those that may not be represented by neither loose nor overtwisted families. We do not know whether tight classes actually exist or whether loose and overtwisted classes might actually coincide in some cases (after conjugating).

\end{document}